\documentclass[reqno]{amsart}

\usepackage{amsfonts}
\usepackage{amssymb}
\usepackage{amsmath}
\usepackage{enumitem}
\usepackage{xcolor}
\usepackage{todonotes}

\setenumerate{label={\rm (\alph{*})}}
\usepackage{bbm,bm,euscript,mathrsfs}
\usepackage{paper_diening}
\usepackage{graphicx}
\usepackage[top=1in, bottom=1.25in, left=1.10in, right=1.10in]{geometry}
\usepackage{hyperref}
\usepackage{cases, color,amsthm}
\usepackage{upref}
\hypersetup{linkcolor=blue, colorlinks=true,citecolor = red}

\allowdisplaybreaks

\numberwithin{equation}{section}

\usepackage{tikz-cd}
\usetikzlibrary{lindenmayersystems}
\usetikzlibrary{decorations.pathreplacing}

\DeclareMathOperator{\Div}{div}

\newcommand{\R}{\mathbb R}
\newcommand{\N}{\mathbb N}
\newcommand{\dd}{\mathrm d}
\newcommand{\dt}{\,\mathrm{d} t}

\newcommand{\vr}{\varrho}
\newcommand{\vt}{\vartheta}

\newcommand{\bu}{\mathbf{u}}

\newcommand{\bn}{\mathbf{n}}
\newcommand{\be}{\mathbf{e}}

\newcommand{\dx}{\, \mathrm{d} {x}}

\newcommand{\nabx}{\nabla }

\newcommand{\Delx}{\Delta }
\newcommand{\Dely}{\Delta_{\mathbf{y}}}

\newcommand{\dH}{\,\mathrm{d}y}
\newcommand{\dxt}{\,\mathrm{d}x\,\mathrm{d}t}
\newcommand{\ds}{\,\mathrm{d}s}

\newcommand{\Oeta}{\Omega_\eta}

\newtheorem{theorem}{Theorem}[section]

\newtheorem{proposition}[theorem]{Proposition}
\newtheorem{corollary}[theorem]{Corollary}
\newtheorem{remark}[theorem]{Remark}

\theoremstyle{definition}
\newtheorem{definition}[theorem]{Definition}

\newcommand{\seb}[1]{\textcolor[rgb]{0.00,0.60,0.20}{  #1}}
\begin{document}

\title[No contact in compressible FSI]{Compressible fluids and elastic plates in 2D: a conditional no-contact theorem}

\author{Dominic Breit}
\address{Institute of Mathematics, TU Clausthal, Erzstra\ss e 1, 38678 Clausthal-Zellerfeld, Germany}
\email{dominic.breit@tu-clausthal.de}
%

\author{Arnab Roy}
\address{Basque Center for Applied Mathematics (BCAM), Alameda de Mazarredo 14, 48009 Bilbao, Spain.}
\email{aroy@bcamath.org}




\date{\today}



\begin{abstract}

We consider the interaction of a compressible fluid with a flexible plate in two space dimensions. The fluid is described by the Navier--Stokes equations in a domain that is changing in accordance with the motion of the structure. The displacement of the latter evolves according to a beam equation. Both are coupled through kinematic boundary conditions and the balance of forces.\\
We prove that for any weak solution to the coupled system, which satisfies certain additional regularity requirements, no contact occurs between the elastic wall and the bottom of the fluid cavity. This applies to both isentropic and heat-conducting fluids.
As a special case of our general theory we extend
the unconditional result from Grandmont and Hillairet (Arch.
Ration. Mech. Anal. 220, 1283--1333, 2016)
on incompressible fluids from visco-elastic to perfectly elastic plates.
\end{abstract}

\maketitle

\section{Introduction}
The interaction of fluids with elastic structures has many applications 
in various fields of applied science such as
 hydro- and aero-elasticity \cite{Do}, bio-mechanics \cite{BGN} and hydrodynamics
\cite{Su}. As a consequence there has been a huge effort from engineers, physicists and mathematicians in studying these complex processes.
In the last two decades the mathematical theory has advanced significantly, still many open questions remain.

We are interested in the case where the elastic structure occupies a part of the boundary of the two-dimensional fluid domain.
To simplify the set-up, we consider the rectangular domain $\Omega= (0,L)\times (0,1)$ with $L>0$. The elastic structure is located on $\omega=(0,L)\times \{1\}$, the upper part of the boundary of $\Omega$, see figure \ref{fig:1}. We set $I=(0,T)$ for some $T>0$. Let $\widehat{\eta}(t,x)$ be the structural displacement, the fluid domain $\Omega_{\eta}\subset \mathbb{R}^2$  is given by:
\begin{equation*}
 \Omega_{\eta} = \{(x,y)\in \mathbb{R}^2 \mid x\in (0,L),\ y\in (0,\eta(t,x))\},
\end{equation*}
where $\eta(t,x)=1+\widehat{\eta}(t,x)$. With some abuse of notation we define  the deformed space-time cylinder $I\times\Omega_\eta=\bigcup_{t\in I}\set{t}\times\Omega_{\eta(t)}\subset\R^{3}$.

The unknown functions for the fluid are the velocity field $\mathbf{u}:I \times \Oeta \rightarrow \mathbb{R}^2$, the density $\varrho:I \times \Oeta \rightarrow \mathbb{R}$, the pressure function $\pi: I \times \Oeta \rightarrow  \mathbb{R}$ and the motion of the fluid is governed by the Navier--Stokes equations
 \begin{equation}\label{2}
\begin{aligned}
 \partial_t (\varrho\bu)  + \Div\big(\varrho\mathbf{u}\otimes\bu\big)
 &= 
 \mu\Delx \bu+\lambda\nabla\Div\bu -\nabx\pi+ \varrho\bff &\text{ for all }(t,x,y)\in I\times\Omega_\eta,\\
 \partial_t\varrho+\Div (\varrho\bu)&=0&\text{ for all }(t,x,y)\in I \times\Omega_\eta,\\
 \end{aligned}
 \end{equation}
where the function $\bff:I \times \Oeta \mapsto\mathbb{R}^2$ is a given volume force and $\mu>0$ and $\lambda\geq0$ are the viscosity coefficients. For a given function $\eta: I \times (0,L) \rightarrow (0,\infty)$, we consider the change of variables $\bm{\varphi}_\eta: I\times\Omega_{\eta_0} \rightarrow I\times\Omega_{\eta}$ given by $\bm{\varphi}_\eta (t,x,y)= (t,x,\tfrac{\eta(t,x)}{\eta(0,x)} y)$. Here $\eta$ solves
 \begin{equation}\label{1}
\begin{aligned}
& \varrho_s\partial_t^2\eta + \alpha\partial_x^4\eta=g- \be_2^\intercal\bm{\tau}\circ\bm{\varphi}_\eta (t,x,\eta(0,x))(-\partial_x\eta \be_1+ \be_2)
&\text{ for all }  (t,x)\in I\times (0,L)
 ,\\
 	\end{aligned}
 \end{equation}
 The parameters $\varrho_s$ and $\alpha$ are positive constants and the function $g: I \times (0,L) \rightarrow  \mathbb{R}$ is a given forcing term. Here $\bftau$ denotes the Cauchy stress of the fluid given by Newton's rheological law, that is
$$\bftau=\mu\big(\nabx\bu+(\nabx\bu)^\intercal-\Div\bu\, \mathbb I_{2\times 2}\big)+(2\mu+\lambda)\Div\bu\, \mathbb I_{2\times 2}-\pi\mathbb I_{2\times 2}.$$
The equations \eqref{2} and \eqref{1} are coupled through the kinematic boundary condition
\begin{align}
\label{interfaceCond}
\bu\circ \bfvarphi_\eta (t,x,\eta(0,x))=\partial_t\eta (t,x)\be_2 \quad\text{ for all } (t,x)\in I\times (0,L).
\end{align} 

\begin{figure}\label{fig:1}
\begin{center}
\begin{tikzpicture}[scale=2]
  \begin{scope}     \draw [thick, <->] (0.5,0.5) -- (0.5,-0.5) -- (2.8,-0.5);
        \node at (0.5,0.6) {$\R$};
        \node at (1.5,-0.25) {$\Omega$};
        \node [blue] at (1.5,0.35) {$\omega=(0,L)\times\{1\}$};
       \node at (3.1,-0.5) {$\R$};
        \draw [blue] (0.5,0.2) -- (2.5,0.2);
                \draw (2.5,0.2) -- (2.5,-0.5);
                   \node at (3.5,0.2) {$\bfvarphi_\eta$};
        \draw [thick, <->] (4.5,0.5) -- (4.5,-0.5) -- (6.8,-0.5);
        \node at (4.5,0.6) {$\R$};
       \node at (7.1,-0.5) {$\R$};
       \draw [blue,dashed] (4.5,0.2) -- (6.5,0.2);
                \draw (6.5,0.2) -- (6.5,-0.5);
                \draw [red] (4.5,0.2) .. controls (4.7,-0.1) and (5,0.8) .. (5.3,0.1);
                \draw [red] (5.3,0.1) .. controls (5.5,-0.1) and  (5.7,0.4) .. (6.5,0.2);
                        \node at (5.4,-0.25) {$\Omega_\eta$};
                         \node [red] at (5.5,0.4) {$\eta$};
          \draw [<-,dashed] (4,-0.2) to [out=150,in=30] (3,-0.2);
  \end{scope}
\end{tikzpicture}
\caption{Domain transformation in 2D.}
\end{center}
\end{figure}
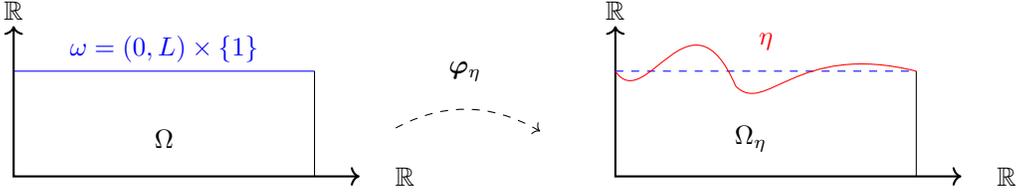
If one aims at the mathematical understanding of  \eqref{1}--\eqref{interfaceCond} beyond a small data/short-time result there is the pressing question:
\begin{align*}
\text{\emph{\textbf{Can contact occur between the elastic wall and the bottom of the fluid cavity?}}}
\end{align*}
 In this case the mathematical model as considered above breaks down as in figure \ref{2fig}. Typically,
global-in-time existence results are only valid up to a possible contact, cf. \cite{LeRu,CanMuh13,MuSc} for incompressible fluids (that is \eqref{2} with constant $\varrho$).
While such a situation is in most cases unnatural from a physical point of view, it seems very difficult to exclude it mathematically in a rigorous manner.
The only available result on the contact problem for elastic structures at the boundary is given in \cite{GraHil}, where an incompressible fluid is interacting with a visco-elastic plate as in figure \ref{fig:1}. In this case the structure equation contains on the left-hand side additionally the dissipative term $-\partial_t\partial_x^2\eta$ rendering it parabolic. Without it a corresponding result remained an open question formulated in \cite[Remark 4]{GraHil}: ``To prove our distance estimate .... One may wonder whether the fluid
dissipation would be sufficient.''

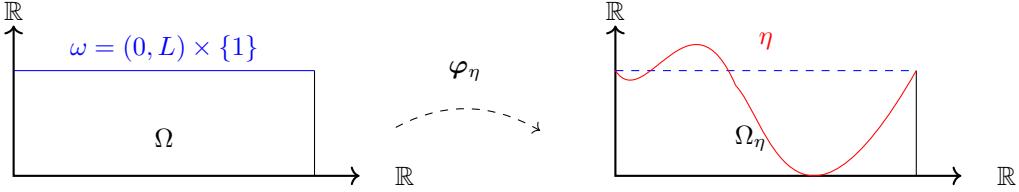
\begin{figure}\label{2fig}
\begin{center}
\begin{tikzpicture}[scale=2]
  \begin{scope}     \draw [thick, <->] (0.5,0.5) -- (0.5,-0.5) -- (2.8,-0.5);
        \node at (0.5,0.6) {$\R$};
        \node at (1.5,-0.25) {$\Omega$};
        \node [blue] at (1.5,0.35) {$\omega=(0,L)\times\{1\}$};
       \node at (3.1,-0.5) {$\R$};
        \draw [blue] (0.5,0.2) -- (2.5,0.2);
                \draw (2.5,0.2) -- (2.5,-0.5);
                   \node at (3.5,0.2) {$\bfvarphi_\eta$};
        \draw [thick, <->] (4.5,0.5) -- (4.5,-0.5) -- (6.8,-0.5);
        \node at (4.5,0.6) {$\R$};
       \node at (7.1,-0.5) {$\R$};
       \draw [blue,dashed] (4.5,0.2) -- (6.5,0.2);
                \draw (6.5,0.2) -- (6.5,-0.5);
                \draw [red] (4.5,0.2) .. controls (4.7,-0.1) and (5,0.8) .. (5.3,0.1);
                \draw [red] (5.3,0.1) .. controls (5.5,-0.1) and  (5.7,-1.2) .. (6.5,0.2);
                        \node at (5.4,-0.25) {$\Omega_\eta$};
                         \node [red] at (5.5,0.4) {$\eta$};
          \draw [<-,dashed] (4,-0.2) to [out=150,in=30] (3,-0.2);
  \end{scope}
\end{tikzpicture}
\caption{Interaction with contact.}
\end{center}
\end{figure}

Much less literature exists in the compressible case. Nevertheless, the existence of weak solutions until the time of contact for shells (see \cite{BrSc,BrScF}, \cite{MMNRS} and \cite{TW}) as well as the existence of local strong solutions (see \cite{MRT,MT,Mi}) has been established rather recently. In this paper we investigate the natural next step which is the question regarding contact. A no-contact result can probably not be expected for weak solutions while it is not of interest for local strong solutions (as the contact can be excluded in short-time if the structure displacement is sufficiently smooth). Hence we aim to prove that there is not contact in finite time provided the weak solution belongs to a certain regularity class. This is a conditional result where the correct class is part of the question. Our main theorem reads in a slightly simplified
version (see Theorem \ref{thm:main} for the statement in its full extend) as follows:
\begin{theorem}
Suppose that $(\eta,\varrho,\bu)$ is a weak solution to \eqref{2}--\eqref{interfaceCond} which satisfies additionally for some $q>\frac{8}{5}$
\begin{align*}
\partial_t(\varrho\bu)&,\,\Div(\varrho\bu\otimes\bu)\in L^1(0,T;L^q(\Omega_\eta)),\, \Div\bu\in L^1(0,T;L^{\infty}(\Omega_\eta)),\\
& \eta \in W^{1,2}(0,T;W^{1,2}_\sharp (0,L))\cap L^\infty(0,T;W^{3,2}_\sharp(0,L)). 
\end{align*} 
Then there is no contact for $t<T$.
\end{theorem}
Notice that we do not need a constitutive law for the pressure $\pi$ such as $\pi=\pi(\varrho)=a\varrho^\gamma$ with $a>0$ and $\gamma>1$. Certainly, these conditions are needed for the existence of solutions, cf.~the references mentioned above. The reason for that is that it is sufficient to work with a divergence-free formulation for the momentum equation in the weak form, cf. Definition \ref{def:weakSolution} (c), which hides the pressure completely. Consequently, our result applies equally to heat-conducting fluids and is independent of the temperature, cf. Section \ref{sec:heat}.

Our approach follows that from \cite{GraHil} which is based on a particular stream function. Its construction relies heavily on the the flat reference geometry for one-dimensional elastic plates. Fortunately, it also suitable for compressible fluids and is completely independent of the pressure. Consequently, we obtain as s special case a statement for incompressible fluids. Given the recent results from \cite{ScSu} on strong solutions for perfectly elastic plates the regularity is not conditional here. Thus we answer the open question from \cite[Remark 4]{GraHil} in the affirmative (see Corollary \ref{cor:2} for the precise statement).

\section{Preliminaries \& main results}
\label{sec:prelim}
\subsection{Conventions}
For simplicity, we set all physical constants in \eqref{2}--\eqref{interfaceCond} to 1. The analysis is not affected as long as they are strictly positive.
For two non-negative quantities $f$ and $g$, we write $f\lesssim g$  if there is a $c>0$ such that $f\leq\,c g$. Here $c$ is a generic constant which does not depend on the crucial quantities. If necessary, we specify particular dependencies. We write $f\approx g$ if both $f\lesssim g$ and $g\lesssim f$ hold. In the notation for function spaces (see the next subsections), we do not distinguish between scalar- and vector-valued functions. However, vector-valued functions will usually be denoted in bold case.

\subsection{Classical function spaces}
Let $\mathcal O\subset\R^2$ be an open set.
The function spaces of continuous or $\alpha$-H\"older-continuous functions, $\alpha\in(0,1]$,
 are denoted by $C(\overline{\mathcal O})$ or $C^{0,\alpha}(\overline{\mathcal O})$ respectively, where $\overline{\mathcal O}$ is the closure of $\mathcal O$. Similarly, we write $C^1(\overline{\mathcal O})$ and $C^{1,\alpha}(\overline{\mathcal O})$.
We denote  by $L^p(\mathcal O)$ and $W^{k,p}(\mathcal O)$ for $p\in[1,\infty]$ and $k\in\mathbb N$, the usual Lebesgue and Sobolev spaces over $\mathcal O$. Spaces of solenoidal functions are denoted by the subscript $\Div$. For $k,m\in\mathbb{N}$, $k<m$ and for $1\leq p,q<\infty$, we consider the definition of Besov spaces by real interpolation of Sobolev spaces 
\begin{equation*}
B^{s}_{q,p}(\mathcal O) = (W^{k,p}(\mathcal O), W^{m,q}(\mathcal O))_{\theta, p}, \mbox{ where }s=(1-\theta)k+\theta m,\ \theta\in (0,1).
\end{equation*}
Note that the space $B^s_{p,p}$ coincides for $s\notin\N$ and $p\in(1,\infty)$ with the fractional Sobolev space (Sobolev-Slobodeckij space) $W^{s,p}(\mathcal O)$. 

On $(0,L)$ we consider periodic function spaces with zero mean value (that is $\int_0^L v\dx=0$)
denoted with subscript $ \sharp$ such as $L^p_ \sharp(0,L)$ and $W^{k,p}_\sharp$ and $B_{ q,p,\sharp}^s$.
We have the following equivalences
\begin{align*}
\|\cdot\|_{W^{1,2}_\sharp(0,L)}\approx \|\partial_x\cdot\|_{L^2(0,L)},\quad \|\cdot\|_{W^{2,2}_\sharp(0,L)}\approx \|\partial_x^2\cdot\|_{L^2(0,L)},\quad \|\cdot\|_{W^{4,2}_\sharp(0,L)}\approx \|\partial_x^4\cdot\|_{L^2(0,L)}.
\end{align*}

For a separable Banach space $(X,\|\cdot\|_X)$, we denote by $L^p(I;X)$ for  $p\in[1,\infty]$, the set of (Bochner-) measurable functions $u:I\rightarrow X$ such that the mapping $t\mapsto \|u(t)\|_{X}$ belongs to $L^p(I)$.
By $C(\overline I;X)$ the space of functions functions $u:I\rightarrow X$ which are continuous with respect to the norm-topology on $X$ and by $C_{\rm w}(\overline I;X)$ the functions which are continuous with respect to the weak topology on $X$. 
Finally, we denote by $W^{1,p}(I;X)$ the subspace of functions from $L^p(I;X)$ whose distributional time derivative belongs to $L^p(I;X)$ as well.

\subsection{Function spaces on variable domains}
\label{ssec:geom}
For a given function $\eta:(0,L)\rightarrow(0,\infty)$, we consider
$\Omega_{\eta}\subset \mathbb{R}^2$ as
\begin{equation*}
 \Omega_{\eta} = \{(x,y)\in \mathbb{R}^2 \mid x\in (0,L),\ y\in (0,\eta(x))\}.
\end{equation*}
The corresponding function spaces for variable domains are defined as follows.
\begin{definition}{(Function spaces)}
For $I=(0,T)$, $T>0$, and $\eta\in C(\overline{I}\times[0,L ],(0,\infty))$ we define for $1\leq p,r\leq\infty$
\begin{align*}
L^p(I;L^r(\Omega_\eta))&:=\big\{v\in L^1(I\times\Omega_\eta):\,\,v(t,\cdot)\in L^r(\Omega_{\eta(t)})\,\,\text{for a.e. }t,\,\,\|v(t,\cdot)\|_{L^r(\Omega_{\eta(t)})}\in L^p(I)\big\},\\
L^p(I;W^{1,r}(\Omega_\eta))&:=\big\{v\in L^p(I;L^r(\Omega_\eta)):\,\,\nabla v\in L^p(I;L^r(\Omega_\eta))\big\}.
\end{align*}
\end{definition}
Similarly, we can define the spaces $W^{1,p}(I;L^{r}(\Omega_\eta))$ and $W^{1,p}(I;W^{1,r}(\Omega_\eta))$ of functions with weak derivatives in time and Sobolev spaces with fractional differentiability (Sobolev-Slobodeckij spaces) $L^{p}(I;W^{s,r}(\Omega_\eta))$ with $s>0$.

\subsection{The main result}
We consider the problem
\begin{align}
\label{eq:1}\partial_t^2\eta + \partial_x^4\eta&=g- \be_2^\intercal\bm{\tau}\circ\bm{\varphi}_\eta (t,x,1)(-\partial_x\eta \be_1+ \be_2)\quad
&\text{ for all }  (t,x)\in I\times (0,L),\\
 \label{eq:2}\partial_t (\varrho\bu)  + \Div\big(\varrho\mathbf{u}\otimes\bu\big)
 &= 
 \mu\Delx \bu+\lambda\nabla\Div\bu -\nabx\pi+ \varrho\bff \quad&\text{ for all }(t,x,y)\in I\times\Omega_\eta,\\
\label{eq:3}\bu\circ \bfvarphi_\eta &=\partial_t\eta (t,x)\be_2 \quad&\text{ for all } (t,x)\in I\times (0,L),
\end{align}
where the Cauchy-stress is given by
$$\bftau=\mu\big(\nabx\bu+(\nabx\bu)^\intercal-\Div\bu\, \mathbb I_{2\times 2}\big)+(2\mu+\lambda)\Div\bu\, \mathbb I_{2\times 2}-\pi\mathbb I_{2\times 2}.$$
We prescribe the periodic boundary condition on the lateral boundary of the container:
\begin{equation}
\bu (t,x,y)=\bu (t,x+L,y),\quad \eta (t,x)=\eta (t,x+L),\quad\text{ for all } (t,x)\in I\times (0,L),\,y\in(0,\eta(x)),
\end{equation}
and no-slip boundary condition at the bottom of the container:
\begin{equation}\label{eq:2410}
\bu (t,x,0)=0\quad\text{ for all } (t,x)\in I\times (0,L).
\end{equation}
The system is completed with the following initial conditions
\begin{equation}\label{ini1}
\eta(0,x)=\eta_0(x)\mbox{ and }\partial_t\eta(0,x)=\eta_*(x) \mbox{ for } x\in (0,L), 
\end{equation}
\begin{equation}\label{ini2}
 \varrho\bu (0,x,y)=\bfm_0(x,y) \mbox{ for }(x,y)\in (0,L)\times (0,\eta_0(x)). 
\end{equation}

Notice that we do neither require a particular relation for the pressure nor is the continuity equation included.
We start with the definition of a weak solution to  \eqref{eq:1}--\eqref{ini2}. Notice that we work with a pressure-free formulation that is sufficient for our main result.
\begin{definition}[Weak solution] \label{def:weakSolution}
Let $(\bff, g, \eta_0, \eta_*, \bu_0)$ be a dataset such that
\begin{equation}
\begin{aligned}
\label{dataset}
&\bff \in L^2\big(I; L^2_{\mathrm{loc}}(\mathbb{R}^3)\big),\quad
g \in L^2\big(I; L^2_\sharp(0,L)\big), \quad
\eta_0 \in W^{2,2}_\sharp(0,L) \text{ with } \min\eta_0>0, 
\\
& 
\eta_* \in L^2_\sharp(0,L), \quad \varrho_0\in L^{\gamma}(\Omega_{\eta_0})\,\,\text{with}\,\,\varrho_0\geq0,\quad \bfm_0\in L^{\frac{2\gamma}{\gamma+1}}(\Omega_{\eta_0}),\\&\qquad\qquad\text{such that }\,\tfrac{\bfm_0}{\varrho_0}\circ\bfvarphi_{\eta_0} =\eta_* \bfe_2\text{ on $(0,L)$}.
\end{aligned}
\end{equation} 
We call the triplet
$(\eta,\varrho,\bu)$
a weak solution to the system \eqref{eq:1}--\eqref{ini2} with data $(\bff, g, \eta_0,  \eta_*,\varrho_0\bu_0)$ provided that the following holds:
\begin{itemize}
\item[(a)] The structure displacement $\eta$ satisfies
\begin{align*}
\eta \in W^{1,\infty} \big(I; L^2_\sharp(0,L) \big)\cap  L^\infty \big(I; W^{2,2}_\sharp(0,L) \big) \quad \text{with} \quad \eta>0 \mbox{ a.e},
\end{align*}
as well as $\eta(0)=\eta_0$ and $\partial_t\eta(0)=\eta_*$.
\item[(b)] The velocity field $\bu \in L^2 \big(I; W^{1,2}(\Omega_{\eta}) \big) \quad \text{with} \quad 
\bu\circ \bfvarphi_\eta =\partial_t\eta (t,x)\be_2\quad\text{for}\quad (t,x)\in I\times(0,L)$.
\item[(c)] The density $\varrho$ is a measurable function and we have
\begin{align*}
\varrho\bfu\in C_w(\overline I;L^1(\Omega_\eta)),\,\varrho\bfu\otimes\bfu\in L^1 \big(I;L^1(\Omega_{\eta}) \big),
\end{align*}
as well as $\varrho\bu(0)=\bfm_0$.
\item[(d)] For all tuple of test-functions $$(\phi, {\bfphi}) \in W^{1,2} \big(I; L^2_\sharp(0,L) \big)\cap  L^2 \big(I; W^{2,2}_\sharp(0,L) \big) \times W^{1,2}(I;L^\infty(\Omega_\eta))\cap L^\infty(I;W^{1,\infty}_{\Div}(\Omega_\eta))$$ with $\phi(T,\cdot)=0$, ${\bfphi}(T,\cdot)=0$ and $\bfphi\circ\bfvarphi_\eta =\phi {\bfe_2}$ on $I\times(0,L)$, we have
\begin{equation*}
\begin{aligned}
&\int_I  \frac{\mathrm{d}}{\dt}\bigg(\int_0^L \partial_t \eta \, \phi \dx
+
\int_{\Oeta}\varrho\bu  \cdot {\bfphi}\,\dd(x,y)
\bigg)\dt 
\\
&=\int_I  \int_{\Oeta}\big(  \varrho\bu\cdot \partial_t  {\bfphi} + \varrho\bu \otimes \bu: \nabla {\bfphi} 
 -  
\mu\nabla \bu:\nabla {\bfphi} +\varrho\bff\cdot{\bfphi} \big) \,\dd(x,y)\dt
\\
&\quad+
\int_I \int_0^L \big(\partial_t \eta\, \partial_t\phi+
 g\, \phi-\partial_x^2\eta\,\partial_x^2 \phi \big)\dx\dt.
 \end{aligned}
\end{equation*}
\end{itemize}
\end{definition}
The system \eqref{eq:1}--\eqref{ini2} is clearly undetermined such that we cannot expect the existence of a solution unless additional assumptions are made. In particular, we will complement \eqref{eq:1}--\eqref{eq:3} with the balance of mass (and a constitutive law for the pressure) in order to obtain the (compressible) Navier--Stokes equations as a special case. Corresponding results follow as corollaries of our main result Theorem \ref{thm:main} and are presented in Section \ref{sec:part} below. 

\begin{theorem}\label{thm:main}
Suppose that $(\eta,\varrho,\bu)$ is a weak solution to the system \eqref{eq:1}--\eqref{ini2} in the sense of Definition \ref{def:weakSolution}. Suppose additionally for some $q>\frac{8}{5}$:
\begin{align}\label{reg:ass}
\begin{aligned}
 \partial_t(\varrho\bu),\,&\Div(\varrho\bu\otimes\bu),\,\varrho\bff\in L^1(0,T;L^q(\Omega_\eta)),\, \Div\bu\in L^1(0,T;L^{\infty}(\Omega_\eta)),\\
& \eta \in W^{1,2}(0,T;W^{1,2}_\sharp (0,L))\cap L^\infty(0,T;W^{3,2}_\sharp(0,L)). 
\end{aligned}
\end{align} 
Then there is no contact for $t<T$.
\end{theorem}
Theorem \ref{thm:main} follows from the distance estimate Proposition \ref{thm:dist} which is proved in Section \ref{sec:3}.

\begin{remark}
If  \eqref{eq:1}--\eqref{ini2} is supplemented with the continuity equation
\begin{align*}
\partial_t\varrho+\Div(\varrho\bfu)=0\quad&\text{ for all }(t,x,y)\in I\times\Omega_\eta
\end{align*}
the standard maximum principle applies, cf. \cite[Theorem 3.3]{BrScF}: Under the assumption that $\Div\bu\in L^1(0,T;L^{\infty}(\Omega_\eta))$, the density stays bounded as long as it is initially. Hence the additional assumption $\varrho\bff\in L^1(0,T;L^q(\Omega_\eta))$ is not needed.
\end{remark}

\begin{remark}
Theorem \ref{thm:main} remains true even if the structure equations contain dissipation, i.e. the term $-\partial_t\partial_x^2\eta$ on the left-hand side or additional second order terms such as $-\partial_x^2\eta$.  
\end{remark}

\section{The distance estimate}\label{sec:3}
In this section, our primary goal is to demonstrate the following result regarding the distance of the structure to the boundary:
\begin{proposition}\label{thm:dist}
Suppose that $(\eta,\varrho,\bu)$ is a weak solution to the system \eqref{eq:1}--\eqref{eq:2410} in the sense of Definition \ref{def:weakSolution}. Suppose additionally for some $q>\frac{8}{5}$:
\begin{align*}
 \partial_t(\varrho\bu),\,&\Div(\varrho\bu\otimes\bu),\,\varrho\bff\in L^1(0,T;L^q(\Omega_\eta)),\, \Div\bu\in L^1(0,T;L^{\infty}(\Omega_\eta)),\\
& \eta \in W^{1,2}(0,T;W^{1,2}_\sharp (0,L))\cap L^\infty(0,T;W^{3,2}_\sharp(0,L)). 
\end{align*} 
Then there exists a constant $C_0$ such that 
\begin{equation*}
\sup_{t\in (0,T)} \|{\eta}^{-1}(t,\cdot)\|_{L^{\infty}_{\sharp} (0,L)} \leq C_0.
\end{equation*}
\end{proposition}
\subsection{The stream function}
 We consider the function
\begin{align}\label{def:psi}
\psi(t,x,y)=\partial_x\eta \psi_0\Big(\frac{y}{\eta(t,x)}\Big),\, \forall\ (t,x,y)\in I \times \Oeta,  
\end{align}
where 
\begin{equation*}
 \psi_0(z)=z^2(3-2z),\ \forall \ z\in (0,1).
\end{equation*}
We define $\bfw:=\curl \psi=(-\partial_y \psi, \partial_x \psi)^\top$ which obviously satisfies
\begin{align*}
\bfw(t,x,\eta(x))&=\partial_x^2\eta(t,x) \be_2\quad (t,x)\in(0,T)\times(0,L),\\
\bfw(t,x,0)&=0\quad (t,x)\in(0,T)\times(0,L).
\end{align*}
Let us recall \cite[Proposition 5]{GraHil} which we need several times in this section:
\begin{proposition}\label{prop8}
Let $\alpha\in (0,\frac{1}{2})\setminus \{\frac{1}{4}\}$. Given a positive function $\eta\in W^{2,2}_\sharp (0,L)$, the following holds:
\begin{equation*}
\int_0^{L}\tfrac{|\partial_x\eta|^4}{|\eta|^{4\alpha}}\dx \lesssim \|\partial_{xx}\eta\|^2_{L^2(0,L)}\|\eta\|^{2(1-2\alpha)}_{L^{\infty}(0,L)} .
\end{equation*}
\end{proposition}
 The following estimates on the derivatives of $\psi$ help us to understand the requirement of the assumption on the integrability in \eqref{reg:ass}, see \eqref{eq:ass1306} below.
\begin{proposition}
Let $\psi$ be defined as in \eqref{def:psi}, $\alpha\in (0,\frac{1}{2})\setminus \{\frac{1}{4}\}$. For $2<p<8/3$, we have
\begin{equation}\label{new:8125}
\|\partial_y\psi(\cdot, t)\|_{L^p(\Omega_\eta)}^p \lesssim \|\partial_x^2\eta\|_{L^2(0,L)}^{p(1-\alpha)}\bigg(1+\int_0^L\tfrac{1}{\eta}\,\dd x\bigg)\quad \forall\ t\in (0,T),
\end{equation}
\begin{equation}\label{new:psix}
\|\partial_x\psi(\cdot, t)\|_{L^p(\Omega_\eta)}^p \lesssim \|\partial_x\eta\|_{L^{\infty}(0,L)}^{2p-4}\|\partial_x^2\eta\|_{L^2(0,L)}^2  +\|\partial_{xx}\eta\|^{p}_{L^p(0,L)}\quad \forall\ t\in (0,T),
\end{equation}
where the hidden constant depends on $\|\eta\|_{L^\infty(0,L)} $
\end{proposition}
\begin{proof}
At first, we determine the partial derivative of the stream function with respect to the spatial variables: For any $(t,x,y)\in I \times \Oeta$ we have
\begin{align*}
\partial_y\psi (t,x,y)&= \tfrac{\partial_x\eta (t,x)}{\eta (t,x)}\psi_0' \left(\tfrac{y}{\eta (t,x)}\right),\\
\partial_x\psi (t,x,y)&= \partial_{xx}\eta (t,x) \psi_0 \left(\tfrac{y}{\eta (t,x)}\right)-\tfrac{|\partial_x\eta (t,x)|^2}{|\eta (t,x)|^2}\psi_0' \left(\tfrac{y}{\eta (t,x)}\right)y.
\end{align*}
We consider $2<p<8/3$. We use H\"{o}lder's inequality and Proposition \ref{prop8} to obtain 
\begin{align*}
\|\partial_y\psi\|_{L^p(\Omega_\eta)}^p&=\int_0^L\int_0^{\eta}\tfrac{|\partial_x\eta(t,x)|^p}{\eta(t,x)^{p}}|\psi_0'(y/\eta(t,x))|^p\,\dd y\dx=\int_0^L\int_0^{1}\tfrac{|\partial_x\eta|^p}{\eta^{p-1}}|\psi_0'(z)|^p\,\dd z\dx\\&\lesssim\int_0^L\tfrac{|\partial_x\eta|^p}{\eta^{\alpha p}}\tfrac{1}{\eta^{p(1-\alpha)-1}}\dx \notag
\\&\lesssim\bigg(\int_0^{L}\tfrac{|\partial_x\eta|^4}{\eta^{4\alpha}}\dx\bigg)^{\frac{p}{4}}\bigg(\int_0^L\tfrac{1}{\eta^{(p(1-\alpha)-1)\frac{4}{4-p}}}\,\dd x\bigg)^{\frac{4-p}{4}} \notag\\&\lesssim \Big(\|\partial_x^2\eta\|_{L^2(0,L)}^2 \|\eta\|_{L^\infty(0,L)}^{2(1-2\alpha)}\Big)^{\frac{p}{4}}\bigg(\int_0^L\tfrac{1}{\eta^{(p(1-\alpha)-1)\frac{4}{4-p}}}\,\dd x\bigg)^{\frac{4-p}{4}}.
\end{align*}
 Observe that $\alpha<1/2$ and $2<p<8/3$ imply that $(p(1-\alpha)-1)\frac{4}{4-p} < 1$. Furthermore, we use Sobolev's embedding and boundedness of $\eta$ to get 
 \begin{align*}
\|\partial_y\psi\|_{L^p(\Omega_\eta)}^p &\lesssim\|\partial_x^2\eta\|_{L^2(0,L)}^{p(1-\alpha)}\bigg(\int_0^L\tfrac{1}{\eta}\,\dd x\bigg)^{\frac{4-p}{4}}\lesssim \|\partial_x^2\eta\|_{L^2(0,L)}^{p(1-\alpha)}\bigg(1+\int_0^L\tfrac{1}{\eta}\,\dd x\bigg).
 \end{align*}
 Similarly, we can compute
\begin{align*}
\|\partial_x\psi\|_{L^p(\Omega_\eta)}^p&\lesssim\int_0^L\int_0^{\eta}\left(\tfrac{|\partial_{x}\eta|^{2p}}{\eta^{2p}}|y|^p |\psi_0'\left(\tfrac{y}{\eta}\right)|^p + |\partial_{xx}\eta|^{p} |\psi_0\left(\tfrac{y}{\eta}\right)|^p\right)\,\dd y\dx \notag  \\
&\lesssim\int_0^L\left(\tfrac{|\partial_x\eta|^{2p}}{\eta^{p-1}}+ |\partial_{xx}\eta|^{p} |\eta| \right)\dx \notag\\
&\lesssim\int_0^L\tfrac{|\partial_x\eta|^{4}}{\eta^{4\alpha }}\tfrac{|\partial_x\eta|^{2p-4}}{\eta^{-4\alpha+p-1}}\dx + \|\eta\|_{L^\infty(0,L)} \|\partial_{xx}\eta\|^{p}_{L^p(0,L)}.
\end{align*}
If we choose $\frac{5}{12}<\alpha<\frac{1}{2}$, then $p< 1+4\alpha$. We use Proposition \ref{prop8} to obtain
\begin{align*}
\|\partial_x\psi\|_{L^p(\Omega_\eta)}^p&\lesssim \|\partial_x\eta\|_{L^{\infty}(0,L)}^{2p-4}\|\eta\|_{L^{\infty}(0,L)}^{4\alpha+1-p}\|\partial_x^2\eta\|_{L^2(0,L)}^2 \|\eta\|_{L^\infty(0,L)}^{2(1-2\alpha)} + \|\eta\|_{L^\infty(0,L)} \|\partial_{xx}\eta\|^{p}_{L^p(0,L)}\notag\\
&\lesssim \|\partial_x\eta\|_{L^{\infty}(0,L)}^{2p-4}\|\partial_x^2\eta\|_{L^2(0,L)}^2 +\|\partial_{xx}\eta\|^{p}_{L^p(0,L)}
\end{align*}
using boundedness of $\eta$ and $4\alpha+1-p+2-4\alpha=3-p\geq0$.
\end{proof}

\subsection{The main estimate}
 We are now going to justify that $(\partial_x^2\eta,\bfw)$ is an admissible test-function in the weak formulation from Definition \ref{def:weakSolution} (d). Note also that $\bfw$ is divergence-free. Note that our assumptions do not imply that we have a strong solution to \eqref{eq:1}--\eqref{ini2}. However, we are now going to argue that the momentum equation is satisfied a.a. in $I\times\Omega_\eta$. Consider an arbitrary point $(t_\star,x_\star)\in I\times\Omega_\eta$ and a parabolic cylinder $I_\star\times B_\star\Subset I\times\Omega_\eta$ containing $(t_\star,x_\star)$. For $\bfphi\in C_{c,\Div}^\infty(I_\star\times B_\star)$, we clearly have
\begin{equation*}
\begin{aligned}
&\int_{I_\star}  \frac{\mathrm{d}}{\dt}
\int_{B_\star}\varrho\bu  \cdot {\bfphi}\,\dd(x,y)
\dt 
\\
&=\int_{I_\star}  \int_{B_\star}\big(  \varrho\bu\cdot \partial_t  {\bfphi} + \varrho\bu \otimes \bu: \nabla {\bfphi} 
 -  
\mathbb S(\nabla \bu):\nabla {\bfphi} +\varrho\bff\cdot{\bfphi} \big) \,\dd(x,y)\dt
 \end{aligned}
\end{equation*}
due to Definition \ref{def:weakSolution} (d) (choosing $\phi=0$). Here 
$$\mathbb S(\nabla\bfu)=\mu\big(\nabx\bu+(\nabx\bu)^\intercal-\Div\bu\, \mathbb I_{2\times 2}\big)+(2\mu+\lambda)\Div\bu\, \mathbb I_{2\times 2}$$
denotes the viscous stress tensor. By assumption we have
\begin{align*}
\partial_t(\varrho\bfu)+\Div\big(\varrho\bfu\otimes\bfu\big)\in L^1(I;L^q( \Omega_\eta))
\end{align*}
for some $q>\frac{8}{5}$
such that we can equivalently write
\begin{equation*}
\begin{aligned}
&\int_{I_\star}\int_{ B_\star}\big(\partial_t(\varrho\bu) +\Div\big(\varrho\bfu\otimes\bfu\big)\big) \cdot {\bfphi}\,\dd(x,y)
\dt 
=\int_{I_\star}  \int_{B_\star}\big(  
 -  
\mathbb S(\nabla \bu):\nabla {\bfphi} +\varrho\bff\cdot{\bfphi} \big) \,\dd(x,y)\dt.
 \end{aligned}
\end{equation*}
Hence, for a.a. $t\in I_\star$
\begin{align*}
\partial_t (\varrho\bu) +\Div\big(\varrho\bfu\otimes\bfu\big)=\Div\mathbb S(\nabla\bfu)+\varrho\bff
\end{align*}
as an equality in $\mathcal D'_{\Div}(B_\star)$. From this we obtain the existence of a distribution $\pi$ on $B_\star$ such that for a.a. $t\in I_\star$
\begin{align*}
\partial_t(\varrho\bu) +\Div\big(\varrho\bfu\otimes\bfu\big)=\Div\mathbb S(\nabla\bfu)-\nabla \pi+\varrho\bff
\end{align*}
in $\mathcal D'_{\Div}(B_\star)$. Since
$\mathbb S(\nabla\bfu)\in L^2(I_\star\times B_\star)$ we must have
$\pi\in L^2(I_\star\times B_\star)$ as well. Note that $\pi$ is unique up to a constant.\footnote{However, this constant will be fixed for the global pressure due its appearance in the structure equation.} 
 Finally we can write
\begin{align}\label{eq:2707a}
\Div\bm{\tau}=\Div\big(\mathbb S(\nabla\bfu)-\pi\,\mathbb I_{2\times 2}\big)=\partial_t(\varrho\bu) +\Div\big(\varrho\bfu\otimes\bfu\big)-\varrho\bff
\end{align}
a.a.  $I_\star\times B_\star$ and, by arbitrariness of $(t_\star,x_\star)$, a.a. in $I\times\Omega_\eta$. In  particular, we  have
\begin{align*}
\Div\bm{\tau}\in L^1(I;L^q(\Omega_\eta)),
\end{align*}
which yields
\begin{align*}
\bm{\tau}\bfn_\eta\in L^1(I;W^{-1/q,q}(\partial\Omega_\eta))
\end{align*}
by the trace theorem \cite[page 51]{So} (note that $\partial\Omega_\eta$ is $C^2$ on account of Sobolev's embedding in 1D and the assumption $\eta\in L^\infty(I;W^{3,2}_\sharp(0,L))$).
Here $W^{-1/q,q}$ denotes the dual of $W^{1/q,q'}$. Now we multiply
\eqref{eq:2707a} by $\bfw$ and integrate by part on the left-hand side obtaining
\begin{align}\label{test1}
\begin{aligned}
\int_0^t\int_{\partial\Omega_\eta}\bm{\tau}\,\bfn_\eta\circ\bm{\varphi}_\eta^{-1}\bfw\,\dd\mathscr H^1\ds&=\mu\int_0^t\int_{\Omega_\eta}\nabla\bu:\nabla\bfw\,\dd(x,y)\ds\\
&+\int_0^t\int_{\Omega_\eta}\Div\bm{\tau}\cdot\bfw\,\dd(x,y)\ds.
\end{aligned}
\end{align}
Here the term on the left-hand side is understood in the distributional sense recalling that $\bfw|_{\partial\Omega}\in L^\infty(I;W^{1,2}(\partial\Omega))$ since $\bfw\circ\bfvarphi_\eta=\partial_t\eta\bfe_2$ on $(0,L)$.
On account of \eqref{eq:2707a} we can decouple the weak formulation obtaining
\begin{equation*}
\begin{aligned}
\int_0^L \partial_t \eta \, \phi \dx\bigg|_{0}^{t}
&=
\int_0^t \int_0^L \big(\partial_t \eta\, \partial_t\phi+
 g\, \phi-\partial_x^2\eta\,\partial_x^2 \phi \big)\dx\ds\\
&+\int_0^t\int_{\partial\Omega_\eta}\bm{\tau}\,\bfn_\eta\circ\bm{\varphi}_\eta^{-1}\phi\circ\bfvarphi_\eta^{-1}\,\dd\mathscr H^1\ds
 \end{aligned}
\end{equation*}
for all $\phi\in W^{1,2}(I;L^{2}_\sharp(0,L))\cap L^{2}(I;W^{2,2}_\sharp(0,L))$.
By smooth approximation (recall that we consider periodic boundary conditions on $(0,L)$) we can insert $\partial_x^2\eta$ such that
\begin{equation}
\begin{aligned}
\int_0^L \partial_t \eta \, \partial_x^2\eta \dx\bigg|_{0}^{t}
&=
\int_0^t \int_0^L \big(-|\partial_t \partial_x\eta|^2+
 g\, \partial_x^2\eta+|\Dely\partial_x\eta|^2\big)\dx\ds\\
&+\int_0^t\int_{\partial\Omega_\eta}\bm{\tau}\,\bfn_\eta\circ\bm{\varphi}_\eta^{-1}\partial_x^2\eta\circ\bfvarphi_\eta^{-1}\,\dd\mathscr H^1\ds.
 \end{aligned}\label{test3}
\end{equation}
Recall that we assume that $\eta\in L^\infty(I;W^{3,2}_\sharp(0,L))$ such that the last term is well-defined
although $\bm{\tau}\,\bfn_\eta\circ\bm{\varphi}_\eta^{-1}$ is only a distribution.
Combining \eqref{test1}--\eqref{test3} we obtain
\begin{align}\nonumber
\int_0^t\int_0^L|\partial_x^3\eta|^2\dx\ds-\int_0^L\partial_t\eta\,\partial_x^2\eta\dx\bigg|_0^t&=\int_0^t\int_0^L|\partial_t\partial_x\eta|^2\dx\ds-\int_0^t\int_0^Lg\,\partial_x^2\eta\dx\ds\\\nonumber
&+\int_0^t\int_{\Omega_\eta}\big(\partial_t(\varrho\bu)+\Div(\varrho\bu\otimes\bu)-\varrho\bff\big)\cdot\bfw\,\dd(x,y)\ds\\& + \mu\int_0^t\int_{\Omega_\eta}\nabla\bu:\nabla\bfw\,\dd(x,y)\ds.\label{eq:2350}
\end{align}
\begin{proof}[Proof of Theorem \ref{thm:dist}]
We rewrite the last term of \eqref{eq:2350} using the function
\begin{align}\label{def:pi}
\pi_\ast(t,x,y):=\partial_{xy}\psi(t,x,y)-\int_0^{x}\partial_y^3\psi(t,z,y)\,\dd z.
\end{align}
This leads us to
\begin{align}\label{nabunabw}
\int_0^t\int_{\Omega_\eta}\nabla\bu:\nabla\bfw\,\dd(x,y)\ds&=\int_0^t\int_{\Omega_\eta} \pi_\ast \Div\bu\,\dd(x,y)\ds+ \int_0^t\int_{\Omega_\eta}\big(\nabla\bfw-\pi_\ast\mathbb I_{2\times 2}\big):\nabla\bu\,\dd(x,y)\ds \notag\\
&= \int_0^t\int_{\Omega_\eta} \pi_\ast \Div\bu\,\dd(x,y)\ds + \int_0^t\int_{\partial\Omega_\eta}\big(\nabla\bfw-\pi_\ast\mathbb I_{2\times 2}\big)\bn_{\eta}\cdot\bu\,\dd\mathscr H^1\ds \notag\\ & -\int_0^t\int_{\Omega_\eta}\big(\Delta\bfw-\nabla\pi_\ast\big)\cdot\bfu\,\dd(x,y)\ds.
\end{align}
Computing
\begin{align*}
\Delta\bfw-\nabla\pi_\ast=\begin{pmatrix} -\partial_x^2\partial_y\psi-\partial_y^3\psi\\
\partial_x^3\psi+\partial_y^2\partial_x\psi
\end{pmatrix}-\begin{pmatrix} \partial_x^2\partial_y\psi-\partial_y^3\psi\\
\partial_x\partial_y^2\psi
\end{pmatrix}=\begin{pmatrix} -2\partial_x^2\partial_y\psi\\
\partial_x^3\psi
\end{pmatrix},
\end{align*}
where we used that
\begin{align}\label{eq:1006}
\begin{aligned}
\int_0^{x}\partial_y^3\psi(t,z,y)\,\dd z&=12\int_0^x\frac{\partial_x\eta}{\eta^3}(t,z)\,\dd z=-6\eta^{-2}(t,z)\Big|_{z=0}^{z=x}\\
&=6\eta^{-2}(t,x)-6\eta_0^{-2}(x)
\end{aligned}
\end{align}
is independent from $y$, we get
\begin{align}\label{eq:1006a}
\begin{aligned}
-\int_0^t\int_{\Omega_\eta}\big(\Delta\bfw-\nabla\pi_\ast\big)\cdot\bfu\,\dd(x,y)\ds&=\int_0^t\int_{\Omega_\eta}\big(2\partial_x^2\partial_y\psi \,u^1-\partial_x^3\psi\,u^2\big)\,\dd(x,y)\ds\\
&=-\int_0^t\int_{\Omega_\eta}\big(2\partial_x^2\psi \,\partial_yu^1-\partial_x^2\psi\,\partial_xu^2\big)\,\dd(x,y)\ds\\
&+\int_0^t\int_{\partial\Omega_\eta}\big(2\partial_x^2\psi \,u^1\,n^2_\eta-\partial_x^2\psi\,u^2\,n^1_\eta\big)\,\dd\mathscr H^1\ds\\
&=-\int_0^t\int_{\Omega_\eta}\partial_x^2\psi\,\big(2\partial_yu^1-\partial_xu^2\big)\,\dd(x,y)\ds\\
&+\int_0^t\int_{0}^L\partial_x^2\psi(s,x,\eta(s,x))\partial_t\eta(s,x)\partial_x\eta(s,x)\,\dd x\ds.
\end{aligned}
\end{align}
Similarly,
using that
\begin{align*}
\nabla\bfw-\pi_\ast\mathbb I_{2\times 2}&=\begin{pmatrix} -\partial_x\partial_y\psi&-\partial_y^2\psi\\
\partial_x^2\psi&\partial_y\partial_x\psi
\end{pmatrix}-\begin{pmatrix} \partial_x\partial_y\psi-\int_0^{x}\partial_y^3\psi(t,z,y)\,\dd z&0\\
0&\partial_x\partial_y\psi-\int_0^{x}\partial_y^3\psi(t,z,y)\,\dd z
\end{pmatrix}\\
&=\begin{pmatrix} -2\partial_x\partial_y\psi+\int_0^{x}\partial_y^3\psi(t,z,y)\,\dd z&-\partial_y^2\psi\\
\partial_x^2\psi&\int_0^{x}\partial_y^3\psi(t,z,y)\,\dd z
\end{pmatrix},
\end{align*}
 it holds
\begin{align}\label{eq:1006b}
\begin{aligned}
\int_0^t\int_{\partial\Omega_\eta}\big(\nabla\bfw-\pi_\ast\mathbb I_{2\times 2}\big)\bn_{\eta}\cdot\bu\,\dd\mathscr H^1\ds&=-\int_0^t\int_0^L\partial_t\eta(s,x)\partial_x\eta(s,x)\partial_x^2\psi(s,x,\eta(s,x))\dx\ds\\
&+\int_0^t\int_0^L\partial_t\eta\bigg(\int_0^x\partial_y^3\psi(t,z,\eta)\,\dd z\bigg)\dx\ds
\end{aligned}
\end{align}
Noting the cancellation from \eqref{eq:1006a} and \eqref{eq:1006b}, we have shown
\begin{multline*}
\int_0^t\int_{\partial\Omega_\eta}\big(\nabla\bfw-\pi_\ast\mathbb I_{2\times 2}\big)\bn_{\eta}\cdot\bu\,\dd\mathscr H^1\ds -\int_0^t\int_{\Omega_\eta}\big(\Delta\bfw-\nabla\pi_\ast\big)\cdot\bfu\,\dd(x,y)\ds \\
=-\int_0^t\int_{\Omega_\eta}\partial_x^2\psi\,\big(2\partial_yu^1-\partial_xu^2\big)\,\dd(x,y)\ds +\int_0^t\int_0^L\partial_t\eta\bigg(\int_0^x\partial_y^3\psi(t,z,\eta)\,\dd z\bigg)\dx\ds .
\end{multline*}
If we use \eqref{eq:1006}, we can rewrite the above relation as 
\begin{align}\nonumber
\int_0^t\int_{\partial\Omega_\eta}&\big(\nabla\bfw-\pi_\ast\mathbb I_{2\times 2}\big)\bn_{\eta}\cdot\bu\,\dd\mathscr H^1\ds -\int_0^t\int_{\Omega_\eta}\big(\Delta\bfw-\nabla\pi_\ast\big)\cdot\bfu\,\dd(x,y)\ds \\
\nonumber&=-\int_0^t\int_{\Omega_\eta}\partial_x^2\psi\,\big(2\partial_yu^1-\partial_xu^2\big)\,\dd(x,y)\ds +6\int_0^L\bigg(\eta^{-1}_0(x)-\eta^{-1}(t,x)\bigg)\dx\\&- \int_0^t\int_0^L\partial_t\eta\eta^{-2}_0(x)\dx\ds \label{2231}
\end{align}
We use the relations \eqref{nabunabw}, \eqref{2231} and the definition of $\pi_{*}$ \eqref{def:pi} in the equation \eqref{eq:2350} and we finally end up with
\begin{align}\nonumber
6\mu\int_0^L\eta^{-1}(t,x)\dx
&=6\mu\int_0^L\eta^{-1}_0(x)\dx - \mu\int_0^t\int_0^L\partial_t\eta\eta^{-2}_0(x)\dx\ds \\\nonumber&+\int_0^t\int_0^L|\partial_x^3\eta|^2\dx\ds-\int_0^L\partial_t\eta\,\partial_x^2\eta\dx\bigg|_0^t  -\int_0^t\int_0^L|\partial_t\partial_x\eta|^2\dx\ds\\\nonumber& -\int_0^t\int_{\Omega_\eta}\big(\partial_t(\varrho\bu)+\Div(\varrho\bu\otimes\bu)-\varrho\bff\big)\cdot\bfw\,\dd(x,y)\ds \\\nonumber&- \mu\int_0^t\int_{\Omega_\eta}\partial_x^2\psi\,\big(2\partial_yu^1-\partial_xu^2\big)\,\dd(x,y)\ds \\&+ \mu\int_0^t\int_{\Omega_\eta}\Big(\partial_{xy}\psi\Div\bu + 6(\eta^{-2}(t,x)-\eta_0^{-2}(x))\Div\bu \Big)\,\dd(x,y)\ds  \nonumber\\
&=:({\tt I})+\dots+({\tt VIII}).\label{eq:final}
\end{align}
We know from the Definition \ref{def:weakSolution} that
$$\eta\in W^{1,\infty}(0,T;L^2_\sharp(0,L))\cap L^\infty(0,T;W^{2,2}_\sharp(0,L)).$$
It ensures that we can bound $({\tt II})$ and $({\tt IV})$ on the right-hand side of \eqref{eq:final}. Moreover, in order to bound $({\tt III})$ and $({\tt V})$ on the right-hand side of \eqref{eq:final}, we need
\begin{align}\label{reg-eta}
\eta\in L^{\infty}(0,T;W^{3,2}_\sharp(0,L)))\cap W^{1,2}(0,T;W^{1,2}_\sharp(0,L)),
\end{align}
which is part of our assumption.
We consider the function $\psi$ as in \eqref{def:psi}. We have 
\begin{equation*}
\partial_y \psi (t,x,y)= \frac{\partial_x \eta (t,x)}{\eta(t,x)} \psi'_0 \Big(\frac{y}{\eta(t,x)}\Big) ,
\end{equation*}
\begin{equation*}
\partial_{xy} \psi (t,x,y)= \frac{\partial_{xx} \eta (t,x)}{\eta(t,x)} \psi'_0 \Big(\frac{y}{\eta(t,x)}\Big)-\frac{|\partial_x \eta(t,x)|^2}{|\eta(t,x)|^2}\psi'_0 \Big(\frac{y}{\eta(t,x)}\Big)- \frac{|\partial_x \eta(t,x)|^2}{|\eta(t,x)|^3} y \psi''_0  \Big(\frac{y}{\eta(t,x)}\Big) .
\end{equation*}
Observe that 
\begin{align}\label{psixy}
&\int_{\Omega_\eta} |\partial_{xy}\psi| \,\dd(x,y) \notag\\&\leq \int_{0}^{L}\int_{0}^{\eta(t,x)} \left|\frac{\partial_{xx} \eta (t,x)}{\eta(t,x)} \psi'_0 \Big(\frac{y}{\eta(t,x)}\Big)\right|\,\dd(x,y) + \int_{0}^{L}\int_{0}^{\eta(t,x)} \left|\frac{|\partial_x \eta(t,x)|^2}{|\eta(t,x)|^3}y \psi''_0 \Big(\frac{y}{\eta(t,x)}\Big)\right|\,\dd(x,y) \notag\\ &+ \int_{0}^{L}\int_{0}^{\eta(t,x)} \left|\frac{|\partial_x \eta(t,x)|^2}{|\eta(t,x)|^2}\psi'_0 \Big(\frac{y}{\eta(t,x)}\Big)\right|\,\dd(x,y) \notag\\ &\lesssim\int_{0}^{L}\Big( \left|{\partial_{xx} \eta (t,x)} \right|+\frac{|\partial_x \eta(t,x)|^2}{|\eta(t,x)|}\Big) \,\dd x .
\end{align}
We use the pointwise estimate
\begin{equation}\label{prop6}
|\partial_x\eta(x)|^2 \leq C \|\partial_{xxx}\eta\|_{L^2 (0,L)}|\eta(x)|,\quad \forall\ x\in (0,L),
\end{equation}
which follows from the one dimensional Sobolev embedding $W^{2,2}_\sharp(0,L)\hookrightarrow L^\infty_\sharp(0,L)$.
Combining \eqref{psixy} and \eqref{prop6}, we conclude
\begin{equation}
\int_{\Omega_\eta} |\partial_{xy}\psi| \,\dd(x,y) \leq C (\|\partial_{xx} \eta\|_{L^2(0,L)}+\|\partial_{xxx} \eta\|_{L^2(0,L)}).
\end{equation}
Now using the regularity of $\eta$ in \eqref{reg-eta} and $\Div\bu\in L^1(0,T;L^{\infty}(\Omega_\eta))$, the first term in $({\tt VIII})$ on the right-hand side of \eqref{eq:final}, i.e. $\int_0^t\int_{\Omega_\eta}\partial_{xy}\psi\Div\bu\,\dd (x,y)\ds$ is bounded. Hence it holds
\begin{align*}
({\tt VIII})&\lesssim\,1+\int_0^t\|\Div\bfu\|_{L^\infty(\Omega_\eta)}\bigg(1+\int_0^L\frac{1}{\eta}\dx\bigg)\ds\\
&\lesssim\,1+\int_0^t\|\Div\bfu\|_{L^\infty(\Omega_\eta)}\bigg(\int_0^L\frac{1}{\eta}\dx\bigg)\ds
\end{align*}
using that $\min\eta_0>0$.
The regularity \eqref{reg-eta} of $\eta$ also implies that $\psi\in L^2(0,T;W^{2,2}(\Omega_\eta))$ by \cite[equ. (128)]{GraHil}. This, together with $\bfu\in L^2(I;W^{1,2}(\Omega_\eta))$ allows to control $({\tt VII})$.

Since $\bfw=(-\partial_y \psi, \partial_x \psi)$, we use the estimates \eqref{new:8125}--\eqref{new:psix} to obtain
\begin{align}\label{eq:ass1306}
\|\bfw\|_{L^p(\Omega_\eta)}^p=\|\partial_y \psi\|_{L^p(\Omega_\eta)}^p + \|\partial_x \psi\|_{L^p(\Omega_\eta)}^p \lesssim \bigg(1+\int_0^L\frac{1}{\eta}\dx\bigg).
\end{align}
Thus
\begin{align*}
({\tt VI})&=\bigg|\int_0^t\int_{\Omega_\eta}\big(\partial_t(\varrho\bu)+\Div(\varrho\bu\otimes\bu)-\varrho\bff\big)\cdot\bfw\,\dd(x,y)\ds\bigg|\\&\lesssim \int_0^t\|\partial_t(\varrho\bu)+\Div(\varrho\bu\otimes\bu)-\varrho\bff\|_{L^{p'}(\Omega_\eta)}\bigg(1+\int_0^L\frac{1}{\eta}\dx\bigg)\ds\\
&\lesssim 1+\int_0^t\|\partial_t(\varrho\bu)+\Div(\varrho\bu\otimes\bu)-\varrho\bff\|_{L^{p'}(\Omega_\eta)}\bigg(\int_0^L\frac{1}{\eta}\dx\bigg)\ds
\end{align*}
Observe that the above integral can be handled provided that
\begin{align*}
\partial_t(\varrho\bu)+\Div(\varrho\bu\otimes\bu)-\varrho\bff\in L^1(0,T;L^{p'}(\Omega_\eta))\mbox{  where  }\frac{8}{5}< p' < 2.
\end{align*}
Now by Gronwall's lemma, we conclude that there exists a constant $C_0$ such that 
\begin{equation*}
\sup_{t\in (0,T)} \|{\eta}^{-1}(t,\cdot)\|_{L^1_{\sharp} (0,L)} \leq C_0.
\end{equation*}
Moreover, we know from \cite[Proposition 7]{GraHil} that there exists a continuous function $D_{\min}: (0,\infty)\times \mathbb{R}^{+}\rightarrow (0,\infty)$ such that
\begin{equation*}
\sup_{t\in (0,T)} \|{\eta}^{-1}(t,\cdot)\|_{L^{\infty}_{\sharp} (0,L)} \leq D_{\min}\left(\|{\eta}^{-1}(t,\cdot)\|_{L^1_{\sharp} (0,L)},\ \|{\eta}(t,\cdot)\|_{H^2_{\sharp} (0,L)}\right).
\end{equation*}
It concludes the proof of Proposition \ref{thm:dist}.
\end{proof}
\section{Particular fluid models}\label{sec:part}
\subsection{Isentropic compressible fluids}
We consider the problem
\begin{align}
\label{eq:1a}\partial_t^2\eta + \partial_x^4\eta&=g- \be_2^\intercal\bm{\tau}\circ\bm{\varphi}_\eta(-\partial_x\eta \be_1+ \be_2)\quad
&\text{ for all }  (t,x)\in I\times (0,L),\\
 \label{eq:2a}\partial_t (\varrho\bu)  + \Div\big(\varrho\mathbf{u}\otimes\bu\big)
 &= 
 \mu\Delx \bu+\lambda\nabla\Div\bu -a\nabx \varrho^\gamma+ \varrho\bff \quad&\text{ for all }(t,x,y)\in I\times\Omega_\eta,\\
 \partial_t\varrho+\Div (\varrho\bu)&=0&\text{ for all }(t,x,y)\in I \times\Omega_\eta,\\
\label{eq:3a}\bu\circ \bfvarphi_\eta&=\partial_t\eta\be_2 \quad&\text{ for all } (t,x)\in I\times (0,L),
\end{align}
where $a>0$, $\gamma>1$ and the Cauchy-stress is given by
\begin{align*}
\bftau=\mu\big(\nabx\bu+(\nabx\bu)^\intercal-\Div\bu\, \mathbb I_{2\times 2}\big)+(2\mu+\lambda)\Div\bu\, \mathbb I_{2\times 2}-a\varrho^\gamma\mathbb I_{2\times 2}.
\end{align*}
\begin{definition}[Weak solution] \label{def:weakSolutionCNS}
Let $(\bff, g, \eta_0, \eta_*, \varrho_0,\bu_0)$ be a dataset such that \eqref{dataset} holds.
We call the triplet
$(\eta,\varrho,\bu)$
a weak solution to the system \eqref{eq:1a}--\eqref{eq:3a} with data $(\bff, g, \eta_0,  \eta_*,\varrho,\bu_0)$ provided that the following holds:
\begin{itemize}
\item[(a)] The structure displacement $\eta$ satisfies
\begin{align*}
\eta \in W^{1,\infty} \big(I; L^2_\sharp(0,L) \big)\cap  L^\infty \big(I; W^{2,2}_\sharp(0,L) \big) \quad \text{with} \quad \eta >0 \mbox{ a.e},
\end{align*}
as well as $\eta(0)=\eta_0$ and $\partial_t\eta(0)=\eta_*$.
\item[(b)] The velocity field $\bu$ satisfies
\begin{align*}
 L^2 \big(I; W^{1,2}(\Omega_{\eta}) \big) \quad \text{with} \quad 
\bu\circ\bfvarphi_\eta =\partial_t \eta {\bfe_2}\quad\text{on}\quad I\times(0,L).
\end{align*}
\item[(c)] The density $\varrho$ satisfies
\begin{align*}
\varrho\in C_{\mathrm w}(\overline I;L^\gamma(\Omega_\eta))
\end{align*}
as well as $\varrho(0)=\varrho_0$.
\item[(d)] The momentum $\varrho\bfu$ satisfies
\begin{align*}
\varrho\bfu\in C_{\mathrm w}(\overline I;L^{\frac{2\gamma}{\gamma+1}}(\Omega_\eta))
\end{align*}
as well as $\varrho\bfu(0)=\bfm_0$.
\item[(e)] For all $\psi\in C^\infty(\overline{I}\times\R^3)$ we have \begin{align}\label{eq:apvarrho0final}
\begin{aligned}
&\int_I\frac{\dd}{\dt}\int_{\Omega_{\eta}}\varrho \,\psi\dxt-\int_I\int_{\Omega_{\eta}}\Big(\varrho\partial_t\,\psi
+\varrho\bfu\cdot\nabla\psi\Big)\dxt=0.
\end{aligned}
\end{align}
\item[(f)]  For all tuple of test-functions $$(\phi, {\bfphi}) \in W^{1,2} \big(I; L^2_\sharp(0,L) \big)\cap  L^2 \big(I; W^{2,2}_\sharp(0,L) \big) \times W^{1,2}(I;L^\infty(\Omega_\eta))\cap L^\infty(I;W^{1,\infty}(\Omega_\eta))$$ 
with $\phi(T,\cdot)=0$, ${\bfphi}(T,\cdot)=0$ and $\bfphi\circ\bfvarphi_\eta =\phi {\bfe_2}$ on $I\times(0,L)$, we have
\begin{equation*}
\begin{aligned}
&\int_I  \frac{\mathrm{d}}{\dt}\bigg(\int_0^L \partial_t \eta \, \phi \dx
+
\int_{\Oeta}\varrho\bu  \cdot {\bfphi}\,\dd(x,y)
\bigg)\dt 
\\
&=\int_I  \int_{\Oeta}\Big(  \varrho\bu\cdot \partial_t  {\bfphi} + \varrho\bu \otimes \bu: \nabla {\bfphi} 
 -  
\mathbb S(\nabla \bu):\nabla {\bfphi} +a\varrho^\gamma\,\Div \bfphi+\varrho\bff\cdot{\bfphi} \Big) \,\dd(x,y)\dt
\\
&\quad+
\int_I \int_0^L \big(\partial_t \eta\, \partial_t\phi+
 g\, \phi-\partial_x^2\eta\,\partial_x^2 \phi \big)\dx\dt.
 \end{aligned}
\end{equation*}
\end{itemize}
\end{definition}

The following result is proved in \cite{BrSc} for the case of shells. A corresponding result for plates follows accordingly.
\begin{theorem}\label{thm:final}
Let $\gamma>1$.
There is a weak solution $(\eta,\varrho, \bfu)$ to \eqref{eq:1a}--\eqref{eq:3a} in the sense of Definition \ref{def:weakSolutionCNS}. Here, we have $I= (0,T_*)$, with $T_*<T$ only in case $\Omega_\eta(s)$ approaches a self intersection with $s\to T_*$.
The solution satisfies the energy estimate
\begin{align*}
\sup_{t\in I}\int_{\Omega_{\eta}}&\varrho|\bfu|^2\,\dd(x,y)+\sup_{t\in I}\int_{\Omega_{ \eta}}a\varrho^\gamma\,\dd(x,y)+\int_I\int_{\Omega_{\eta}}|\nabla\bfu|^2\dxt+\sup_{t\in I}\int_0^L|\partial_x^2\eta|^2\dx+\sup_{t\in I}\int_0^L|\partial_t\eta|^2\dx\\
\lesssim&\,\int_{\Omega}\frac{|\bfm_0|^2}{\varrho_0}\,\dd(x,y)+\int_{\Omega}a\varrho_0^\gamma\,\dd(x,y)+\int_I\|\bff\|_{L^\infty(\Omega_{\eta})}^2\dt+\int_I\norm{g}_{L^2(0,L)}^2\dt\\
+&\,\int_0^L\abs{\partial_x^2\eta_0}^2\dx+\int_0^L\abs{\eta_{*}}^2\dx,
\end{align*}
provided that $\eta_0,\eta_1,\varrho_0,\bfm_0,\bff$ and $g$ are regular enough to give sense to the right-hand side and that \eqref{dataset} is satisfied.
\end{theorem}

Theorem \ref{thm:main} implies immediately the following
\begin{corollary}\label{cor:1}
Suppose that the data satisfies \eqref{dataset}.
Let $(\eta, \varrho, \bu)$ be a weak solution to \eqref{eq:1a}--\eqref{eq:3a} in the sense of Definition \ref{def:weakSolutionCNS} that satisfies for some
$q>\frac{8}{5}$:
\begin{align}\label{reg:assa}
\begin{aligned}
 \partial_t(\varrho\bu),\,&\Div(\varrho\bu\otimes\bu),\,\varrho\bff\in L^1(0,T;L^q(\Omega_\eta)),\, \Div\bu\in L^1(0,T;L^{\infty}(\Omega_\eta)),\\
& \eta \in W^{1,2}(0,T;W^{1,2}_\sharp (0,L))\cap L^\infty(0,T;W^{3,2}_\sharp(0,L)). 
\end{aligned}
\end{align} 
Then there is no contact for $t<T$.
\end{corollary}
We want to provide some examples of fluid-structure systems where the assumptions \eqref{reg:assa} will fulfil. 
Let the initial data satisfies: \begin{gather*}
\eta_{0} \in W^{4,2}_\sharp(0,L), \quad \min \eta_{0} > 0,\quad \eta_{*} \in W^{3,2}_\sharp(0,L), \label{ini-1}\\
 {\varrho}_{0} \in W^{3,2}(\Omega_{\eta_0}), \quad \min {\varrho}_{0} > 0, \label{ini-2}\\
{\bu}_{0} \in W^{3,2}(\Omega_{\eta_0}), \quad  {\bu}_{0}  = 0 \mbox{ on } (0,L)\times \{0\}, \quad  
{\bu}_{0}\left(x, \eta_0(x)\right) = \eta_{*} (x) e_{2} \quad x\in (0,L). \label{ini-3}
\end{gather*}
Under these assumptions it is shown in \cite[Theorem 1.1]{MRT} that there exists  $T > 0$ such that the system \eqref{eq:2a}--\eqref{eq:3a} with \eqref{eq:1a} replaced by
\begin{equation*}
\partial_t^2\eta - \partial_x^2\eta=g- \be_2^\intercal\bm{\tau}\circ\bm{\varphi}_\eta(-\partial_x\eta \be_1+ \be_2) \quad
\text{ for all }  (t,x)\in I\times (0,L)
\end{equation*}
admits a unique strong solution $(\eta,{\varrho}, {\bfu})$ satisfying
\begin{equation*}
 \varrho \in W^{1,2}(0,T;W^{3,2}(\Omega_\eta))\cap W^{1,\infty}(0,T;W^{2,2}(\Omega_\eta)),
\end{equation*}
\begin{equation*}
 \bu \in L^{2}(0,T;W^{4,2}(\Omega_\eta))\cap C([0,T);W^{3,2}(\Omega_\eta)) 
\cap  H^{1}(0,T;W^{2,2}(\Omega_\eta)) 
\cap W^{2,2}(0,T;L^{2}(\Omega_\eta)), 
\end{equation*}
$$
\eta \in  L^{\infty}(0,T;W^{4,2}_\sharp(0,L)) \cap W^{2,2}(0,T;W^{2,2}_\sharp(0,L)) \cap W^{3,2}(0,T;L^{2}_\sharp(0,L)) \cap W^{1,\infty}(0,T;W^{3,2}_\sharp(0,L)).
$$
Clearly, a strong solution satisfies the requirements \eqref{reg:assa}. Similarly, we can also verify that if we consider the system \eqref{eq:2a}--\eqref{eq:3a} with \eqref{eq:1a} replaced by
\begin{equation*}
\partial_t^2\eta + \partial_x^4\eta - \partial_t\partial_x^2\eta=g- \be_2^\intercal\bm{\tau}\circ\bm{\varphi}_\eta(-\partial_x\eta \be_1+ \be_2) \quad
\text{ for all }  (t,x)\in I\times (0,L),
\end{equation*}
then according to \cite[Theorem 1.7]{Mi}, under certain assumptions of initial data, the strong solution $(\eta,{\varrho}, {\bfu})$ satisfies the requirements \eqref{reg:assa}.

\subsection{Heat-conducting compressible fluids}\label{sec:heat}
We consider the problem
\begin{align}
\label{eq:1A}\partial_t^2\eta + \partial_x^4\eta&=g- \be_2^\intercal\bm{\tau}\circ\bm{\varphi}_\eta(-\partial_x\eta \be_1+ \be_2)\quad
&\text{ for all }  (t,x)\in I\times (0,L),\\
 \label{eq:2A}\partial_t (\varrho\bu)  + \Div\big(\varrho\mathbf{u}\otimes\bu\big)
 &= 
 \mu\Delx \bu+\lambda\nabla\Div\bu -a\nabx (\varrho\vartheta)+ \varrho\bff \quad&\text{ for all }(t,x,y)\in I\times\Omega_\eta,\\
 \partial_t\varrho+\Div (\varrho\bu)&=0&\text{ for all }(t,x,y)\in I \times\Omega_\eta,\\\label{eq:3A}
c_v \partial_t(\varrho\vartheta)+\Div (\varrho\vartheta\bu)&=\kappa\Delta\vartheta+\bftau:\nabla\bfu+\varrho H&\text{ for all }(t,x,y)\in I \times\Omega_\eta,\\
\label{eq:4A} \bu\circ \bfvarphi_\eta &=\partial_t\eta (t,x)\be_2
 \quad&\text{ for all } (t,x)\in I\times (0,L),\\
 \bu (t,x,0)&=0 \quad&\text{ for all } (t,x)\in I\times (0,L),\\
\label{eq:5A} \nabla\vartheta\cdot\bfn_\eta\circ\bfvarphi_\eta^{-1} &=0\quad&\text{ for all } (t,x,y)\in I\times \partial\Omega_\eta,
\end{align}
where the Cauchy-stress is given by
\begin{align*}
\bftau&=\mu\big(\nabx\bu+(\nabx\bu)^\intercal-\Div\bu\, \mathbb I_{2\times 2}\big)+(2\mu+\lambda)\Div\bu\, \mathbb I_{2\times 2}-a\varrho\vartheta\mathbb I_{2\times 2}\\
&=\mathbb S(\nabla\bfu)-a\varrho\vartheta\mathbb I_{2\times 2}
\end{align*}
and $a,c_v,\kappa$ are positive constants. The function $H:\R^2\rightarrow [0,\infty)$ represents some external heat source.

The system is completed with the following initial conditions
\begin{equation}\label{iniheat1}
\eta(0,x)=\eta_0(x)\mbox{ and }\partial_t\eta(0,x)=\eta_*(x) \mbox{ for } x\in (0,L), 
\end{equation}
\begin{equation}\label{iniheat2A}
 \varrho(0,x,y)=\varrho_0(x,y),\ \bu (0,x,y)=\bu_0(x,y),\ \vartheta (0,x,y)=\vartheta_0 (x,y) \mbox{ for }(x,y)\in (0,L)\times (0,\eta_0(x)). 
\end{equation}
In order to define a weak solution one can proceed as in
Definition \ref{def:weakSolutionCNS} as far as the first three equation are concerned. However, equation \eqref{eq:4A} is considered in terms of the specific entropy $s=\log \vartheta^{c_v}-\log(\varrho)$ whose balance reads as
\begin{align*}
c_v \partial_t(\varrho s)+\Div (\varrho s\bu)&=\frac{1}{\vartheta}\Big(\frac{\kappa|\nabla\vartheta|^2}{\vartheta}+\mathbb S(\nabla\bfu):\nabla\bfu\Big)+\varrho \frac{H}{\vartheta}\quad\text{ for all }(t,x,y)\in I \times\Omega_\eta.
\end{align*}
On level of strong solutions one can substitute \eqref{eq:3A} by it and obtains an equivalent problem.
In the weak formulation, where this is not the case anymore, the specific entropy balance is formulate as an inequality, i.e.,
\begin{align} \label{m217*final1}\begin{aligned}
 \int_I \frac{\dd}{\dt}\int_{\Omega_{\eta}} \vr s \, \psi,\dd(x,y)\dt
 &- \int_I \int_{\Omega_{\eta}} \big( \vr s\partial_t \psi + \varrho s \bfu \cdot \nabla\psi \big)\,\dd (x,y)\dt
\\& \geq\int_I\int_{\Omega_{\eta}}
\frac{1}{\vartheta}\Big(\mathbb S(\nabla\bfu):\nabla\bfu+\frac{\kappa}{\vartheta}|\nabla\vartheta|^2\Big)
\psi\,\dd (x,y)\dt \\
&+ \int_I\int_{\Omega_{\eta}} \frac{\kappa\nabla\vartheta}{\vartheta} \cdot \nabla\psi \,\dd(x,y)\dt+ \int_I \int_{\Omega_{\eta}} \frac{\vr}{\vt} H \psi  \,\dd(x,y)\dt
\end{aligned}
\end{align}
holds for all
$\psi\in C^\infty(\overline I\times \R^3)$ with $\psi \geq 0$.
 Similar to Theorem \ref{thm:final} the existence of a weak solution to \eqref{eq:1A}--\eqref{eq:5A} is shown in \cite{BrScF}, where $\vartheta\in L^2(I;W^{1,2}(\Omega_\eta))$ is non-negative. Note, however, that the constitutive relations for pressure and internal energy differ from those in \eqref{eq:1A}--\eqref{iniheat2A} and the viscosity coefficients must depend on $\vartheta$.

We obtain the following corollary from Theorem \ref{thm:main} which is completely independent form the temperature $\vartheta$.
\begin{corollary}\label{cor:1}
Suppose that the data satisfies \eqref{dataset}, that $\vartheta_0\in L^4(\Omega_{\eta_0})$ and $H\in L^\infty_{\mathrm loc}(I\times\R^2)$ with $\vartheta_0,H\geq0$.
Let $(\eta,\varrho,\bu,\vartheta)$ be a weak solution to \eqref{eq:1A}--\eqref{iniheat2A} that satisfies \eqref{reg:assa}. Then there is no contact for $t<T$.
\end{corollary}
We know the following local-in-time existence result for the system \eqref{eq:1A}--\eqref{iniheat2A} from \cite[Theorem 1.1]{MT}:
Let $2< p<\infty$, $3<q<\infty$, $\frac{1}{p}+\frac{1}{2q}\neq \frac{1}{2}$. Suppose the initial data satisfy:
\begin{gather*}
\eta_{0} \in B^{2(2-1/p)}_{q,p,\sharp}(0,L), \quad \eta_{*} \in B^{2(1-1/p)}_{q,p,\sharp}(0,L), \\
 {\varrho}^{0} \in W^{1,q}(\Omega_{\eta_0}), \quad \min {\varrho}_{0} > 0, \\
{\bu}_{0} \in B^{2(1-1/p)}_{q,p}(\Omega_{\eta_0}),\quad \vartheta_0 \in B^{2(1-1/p)}_{q,p}(\Omega_{\eta_0}),
\end{gather*}
with the compatibility conditions 
\begin{align*}
{\bu}_{0}  = 0 \mbox{ on } (0,L)\times \{0\}, \quad  &
{\bu}_{0}\left(x, \eta_0(x)\right) = \eta_{*} (x) e_{2} \quad x\in (0,L),\\
\nabla \vartheta_0\cdot\bfn_{\eta_0}\circ\bfvarphi_{\eta_0}^{-1}&=0 \mbox{ on }\partial\Omega_{\eta_0}, \quad \mbox{if}\quad \frac{1}{p}+\frac{1}{2q}<\frac{1}{2}.
\end{align*}
Then there exists $T>0$ such that the system \eqref{eq:2A}--\eqref{iniheat2A}
with
\eqref{eq:1A} replaced by
\begin{align*}
\partial_t^2\eta  - \partial_t\partial_x^2\eta+ \partial_x^4\eta&=g- \be_2^\intercal\bm{\tau}\circ\bm{\varphi}_\eta(-\partial_x\eta \be_1+ \be_2)\quad
\text{ for all }  (t,x)\in I\times (0,L)
\end{align*}
 admits a unique strong solutions satisfying 
\begin{align*}
 \varrho \in W^{1,p}&(0,T;W^{1,q}(\Omega_\eta)),\quad \bu \in  L^{p}(0,T;W^{2,q}(\Omega_\eta)) \cap W^{1,p}(0,T;L^{q}(\Omega_\eta)),\\
 &\vartheta \in L^{p}(0,T;W^{2,q}(\Omega_\eta)) \cap W^{1,p}(0,T;L^{q}(\Omega_\eta)),\\ \eta\in L^{p}&(0,T;W^{4,q}_\sharp(0,L)) \cap W^{1,p}(0,T;W^{2,q}_\sharp(0,L)) \cap W^{2,p}(0,T;L^{q}_\sharp(0,L)).
\end{align*}
Observe that a strong solution of the system \eqref{eq:1A}--\eqref{iniheat2A} satisfies the requirements \eqref{reg:assa}.

\subsection{Incompressible fluids}
We consider the problem
\begin{align}
\label{eq:1b}\partial_t^2\eta + \partial_x^4\eta&=g- \be_2^\intercal\bm{\tau}\circ\bm{\varphi}_\eta(-\partial_x\eta \be_1+ \be_2)\quad
&\text{ for all }  (t,x)\in I\times (0,L),\\
 \label{eq:2b}\partial_t \bu  + \Div\big(\mathbf{u}\otimes\bu\big)
 &= 
 \mu\Delx \bu -\nabx \pi+ \bff \quad&\text{ for all }(t,x,y)\in I\times\Omega_\eta,\\
 \Div \bu&=0&\text{ for all }(t,x,y)\in I \times\Omega_\eta,\\
\label{eq:3b}\bu\circ \bfvarphi_\eta &=\partial_t\eta (t,x)\be_2 \quad&\text{ for all } (t,x)\in I\times (0,L),\\
\bu (t,x,0)&=0 \quad&\text{ for all } (t,x)\in I\times (0,L),
\end{align}
where the Cauchy-stress is given by
$$\bftau=\mu\big(\nabx\bu+(\nabx\bu)^\intercal\big)-\pi\mathbb I_{2\times 2}.$$
The system is completed with the following initial conditions
\begin{equation}\label{iniheat1}
\eta(0,x)=\eta_0(x)\mbox{ and }\partial_t\eta(0,x)=\eta_*(x) \mbox{ for } x\in (0,L), 
\end{equation}
\begin{equation}\label{iniheat2}
  \bu (0,x,y)=\bu_0(x,y) \mbox{ for }(x,y)\in (0,L)\times (0,\eta_0(x)). 
\end{equation}
We speak about a strong solution to \eqref{eq:1b}--\eqref{eq:3b} if all quantities are $L^2$-function in space-time, i.e., if it holds
\begin{align}\label{eq:1306}
\begin{aligned}
\bu&\in L^2(I;W^{2,2}(\Omega_\eta))\cap W^{1,\infty}(I;L^2(\Omega_\eta)),\,\,
\pi\in L^2(I;W^{1,2}(\Omega_\eta)),\\
\eta&\in W^{2,\infty}(I;L^{2}_\sharp(0,L))\cap W^{1,\infty}(I;W^{2,2}_\sharp(0,L))\cap L^\infty(I;W^{4,2}_\sharp(0,L)).
\end{aligned}
\end{align}
The existence of a strong solution to \eqref{eq:1b}--\eqref{iniheat2} that satisfies \eqref{eq:1306} has been shown very recently in \cite{ScSu} under the assumption that no collision occurs. A corresponding result for visco-elastic plates (if the term $-\partial_t\partial_x^2\eta$ is added to the right-hand side of \eqref{eq:1b}) was shown before in \cite{GraHil}, where also the contact could be excluded.
As a special case of Theorem \ref{thm:main} (setting $\varrho=1$) we obtain now that also for perfectly elastic plates, there is no collision. Noticing that \eqref{eq:1306} implies that \eqref{reg:ass} holds we obtain the following unconditional result:
\begin{corollary}\label{cor:2}
Let $T>0$ be arbitrary. Suppose that the data satisfies\footnote{These are the assumptions from \cite{ScSu} ensuring the existence of a strong solution.}
\begin{equation*}
\eta_0\in W^{4,2}_\sharp(0,L), \quad \eta_{*}\in W^{2,2}_\sharp(0,L), \quad \bu_0\in W^{2,2}(\Omega_{\eta_0})
\end{equation*}
with the following compatibility conditions:
\begin{align*}
{\bu}_{0}  = 0 \mbox{ on } (0,L)\times \{0\}, \quad  &
{\bu}_{0}\left(x,  \eta_0(x)\right) = \eta_{*} (x) e_{2} \quad x\in (0,L),\\
\Div\bu_0=0 \mbox{ in } \Omega_{\eta_0}, &\quad 
\min \eta_0>0.
\end{align*}
There is a strong solution $(\eta,\bu,\pi)$ to \eqref{eq:1b}--\eqref{iniheat2} that satisfies \eqref{eq:1306} and there is no contact for $t<T$.
\end{corollary}
\begin{remark}
In connection with Corollary \ref{cor:2} let us formulate three open problems:\footnote{They are clearly also open in the compressible case.}
\begin{enumerate}
\item Is it possible to obtain a conditional no-contact result in the 3D case? The existence of a global strong solution is clearly out of reach but one may hope for a distance estimate taking it for granted. However, the approach here is heavily based on the stream function and it is unclear how to argue in the 3D case.
\item Is it possible to obtain a version
of Corollary \ref{cor:2} for elastic shells, where the reference-geometry is not flat? In the case of a visco-elastic fluid the existence of a strong solution is shown \cite{Br}, but it is unclear if one can obtain an estimate as in Section \ref{sec:3}.
\item Is it possible to obtain a version
of Corollary \ref{cor:2} for inviscid fluids? For the estimate in Section \ref{sec:3} the viscosity of the fluid is crucial.
\end{enumerate}
\end{remark}

\section*{Acknowledgments}
\smallskip
\par\noindent
{\bf Funding}. This research was partly funded by:
\\
(i) Grant BR 4302/3-1 (525608987) by the German Research Foundation (DFG) within the framework of the priority research program SPP 2410 (D. Breit);\\
(ii) Grant BR 4302/5-1 (543675748) by the German Research Foundation (DFG) (D. Breit);\\
(iii) This research is supported by the Basque Government through the BERC 2022-2025 program and by the Spanish State Research Agency through BCAM Severo Ochoa excellence accreditation CEX2021-01142-S funded by MICIU/AEI/10.13039/501100011033 and through Grant PID2023-146764NB-I00 funded by MICIU/AEI/10.13039/501100011033 and cofunded by the European Union (A. Roy);\\
(iv) Grant RYC2022-036183-I funded by MICIU/AEI/10.13039/501100011033 and by ESF+ (A. Roy).

\section*{Compliance with Ethical Standards}
\smallskip
\par\noindent
{\bf Conflict of Interest}. The authors declare that they have no conflict of interest.

\smallskip
\par\noindent
{\bf Data Availability}. Data sharing is not applicable to this article as no datasets were generated or analysed during the current study.

\end{document}